\documentclass{article}

\usepackage{amsmath}
\usepackage{amssymb}
\usepackage{amsfonts}
\usepackage{amsthm}
\usepackage{color}
\usepackage{graphicx}
\usepackage{hyperref}
\usepackage{ulem}
\usepackage{tikz}
\usetikzlibrary{calc}
\usepackage{mathabx}
\theoremstyle{plain}
\newtheorem{thm}{Theorem}[section]
\newtheorem{lem}[thm]{Lemma}
\newtheorem{prop}[thm]{Proposition}

\newtheorem{cor}[thm]{Corollory}
\theoremstyle{definition}

\theoremstyle{remark}
\newtheorem{remark}[thm]{Remark}

\newtheorem{conj}[thm]{Conjecture}

\begin{document}

\title {\bf Periodicity related to a sieve method of producing primes}

\author{\it Haifeng Xu\thanks{Project supported by NSFC(Grant No. 11401515), the University Science Research Project of Jiangsu Province (14KJB110027) and the Foundation of Yangzhou University 2014CXJ004.},\quad Zuyi Zhang,\quad Jiuru Zhou}
\date{\small\today}

\maketitle

\begin{abstract}
In this paper we consider a slightly different sieve method from Eratosthenes' to get primes. We find the periodicity and mirror symmetry of the pattern. 
\end{abstract}



\section{Definitions and Notations}

Let $\{p_1,p_2,\ldots,p_n,\ldots\}$ be the set of all primes, where $p_n$ denotes the $n$-th prime number. $\pi(x)$ is the number of primes which are less than or equal to $x$. $\mathbb{N}$ is the set of positive numbers. Let
\[
M_{p_n}=\mathbb{N}-\{2k,3k,5k,\ldots,p_n k\mid k\in\mathbb{N}\},
\]
and $D_{p_n}$ be the set of difference of two consecutive numbers in $M_{p_n}$, that is,
\[
D_{p_n}=\{d_k\mid d_k=x_{k+1}-x_k, x_i\in M_{p_n}\},
\]
where $x_i$ is the $i$-th number in $M_{p_n}$.

For example, we list the first few numbers of $M_3$ and $D_3$.

\begin{tabular}{|c|c|c|c|c|c|c|c|c|c|c|c|c|c|c|}
\hline
$M_3$ & 1 & 5 & 7 & 11 & 13 & 17 & 19 & 23 & 25 & 29 & 31 & 35 & 37 & $\ldots$\\
\hline
$D_3$ & 4 & 2 & 4 & 2  & 4  & 2  & 4  & 2  & 4  & 2  & 4  & 2  & 4  & $\ldots$\\
\hline
\end{tabular}

We call the minimum subset $\mathcal{P}_3:=\{4,2\}$ of $D_3$ as the pattern of $D_3$ since it occurs periodically. We will prove this fact in the next section and also for the general situation. Hence, we put the general definitions of $P_{p_n}$ and $T_{p_n}$ in the proof of Theorem \ref{thm:periodicity} due to the induction.

The number of the elements in the pattern $\mathcal{P}_3$ for $D_3$ is called the period of $D_3$. We write it as $T_3=2$.

The length of the pattern $\mathcal{P}_{p_n}$ is defined as the sum of the elements in the pattern. We write it as $L(\mathcal{P}_{p_n})$.

\section{Periodicity of the pattern}

\begin{prop}\label{prop:1}
$M_3=\{u_n\}_{n=1}^{\infty}$ has the following property:
\[
\begin{cases}
u_{2k}=u_{2k-1}+4,\\
u_{2k+1}=u_{2k}+2,\\
\end{cases}
\]
where $k=1,2,\ldots$. In other words, $\{4,2\}$ is the pattern of $D_3$, so the period is $2$.
\end{prop}
\begin{proof}
$u_n$ satisfies the equations:
\[
\begin{cases}
x\equiv 1(\text{mod}\ 2),\\
x\equiv 1(\text{mod}\ 3),\\
\end{cases}
\quad\text{or}\quad
\begin{cases}
x\equiv 1(\text{mod}\ 2),\\
x\equiv 2(\text{mod}\ 3).\\
\end{cases}
\]
The left infers that $x\equiv 1(\text{mod}\ 6)$. The right is:
\[
\begin{cases}
x=2k+1,\\
x=3\ell+2.\\
\end{cases}
\]
If $u_n=2k+1=3\ell+2$, then $2k=3\ell+1$, which infers that $\ell$ must be odd. Let $\ell=2h+1$, then $2k=3(2h+1)+1=6h+4$. Thus $2k+1=6h+5=3\ell+2$. Hence, $u_n\equiv 5(\text{mod}\ 6)$.
\end{proof}

\begin{thm}\label{thm:periodicity}
The period of $D_p$ is
\[
T_p=(2-1)(3-1)(5-1)\cdots(p-1).
\]
\end{thm}
\begin{proof}
We prove this by induction. The case for $D_3$ has been proved. For simplicity, we try to explain the procedure by proving the case for $D_5$. First, we observe that
\[
M_{p_{n+1}}=M_{p_n}-\{p_{n+1}h\mid h\in M_{p_n}\}.
\]
We show this procedure (getting $M_5$ from $M_3$) in the following table (see Figure \ref{fig:M5}).\\

\begin{figure}[htbp]
\centering
\begin{tabular}{|c|c|c|c|c|c|c|c|c|c|c|}
\hline
1& {\bf\color{red}\sout{5}} & 7 & 11 & 13 & 17 & 19 & 23 & {\bf\color{blue}\sout{25}} & 29 & 31 \\
 & {\bf\color{red}\sout{35}} & 37 & 41 & 43 & 47& 49 & 53 & {\bf\color{blue}\sout{55}} & 59 & 61 \\
 & {\bf\color{red}\sout{65}} & 67 & 71 & 73 & 77 & 79 & 83 & {\bf\color{blue}\sout{85}} & 89 & 91 \\
 & {\bf\color{red}\sout{95}} & 97 & 101 & 103 & 107 & 109 & 113 & {\bf\color{blue}\sout{115}} & 119 & 121 \\
 & {\bf\color{red}\sout{125}} & 127 & 131 & 133 & 137& 139 & 143 & {\bf\color{blue}\sout{145}} & 149 & 151 \\
 & {\bf\color{red}\sout{155}} & 157 & 161 & 163 & 167 & 169 & 173 & {\bf\color{blue}\sout{175}} & 179 & 181\\
 & {\bf\color{red}\sout{185}} & 187 & 191 & 193 & 197 & 199 & 203 & {\bf\color{blue}\sout{205}} & 209 & 211\\
\hline
 & {\bf\color{red}\sout{215}} & 217 & 221 & 223 & 227 & 229 & 233 & {\bf\color{blue}\sout{235}} & 239 & 241\\
 & {\bf\color{red}\sout{245}} & 247 & 251 & 253 & 257 & 259 & 263 & {\bf\color{blue}\sout{265}} & 269 & 271\\
 & {\bf\color{red}\sout{275}} & 277 & 281 & 283 & 287 & 289 & 293 & {\bf\color{blue}\sout{295}} & 299 & 301\\
 & {\bf\color{red}\sout{305}} & 307 & 311 & 313 & 317 & 319 & 323 & {\bf\color{blue}\sout{325}} & 329 & 331\\
 & {\bf\color{red}\sout{335}} & 337 & 341 & 343 & 347 & 349 & 353 & {\bf\color{blue}\sout{355}} & 359 & 361\\
 & {\bf\color{red}\sout{365}} & 367 & 371 & 373 & 377 & 379 & 383 & {\bf\color{blue}\sout{385}} & 389 & 391\\
 & {\bf\color{red}\sout{395}} & 397 & 401 & 403 & 407 & 409 & 413 & {\bf\color{blue}\sout{415}} & 419 & 421\\
 \hline
 & $\cdots$ & & & & & & & & & \\
\end{tabular}
\caption{$M_5$}
\label{fig:M5}
\end{figure}

We need to delete the multiples of $5$ in $M_3$, i.e.,
\[
\begin{split}
&{\bf\color{red}5\cdot 1},\quad 5\cdot 5,\\
&{\bf\color{red}5\cdot 7},\quad 5\cdot 11,\\
&{\bf\color{red}5\cdot 13},\quad 5\cdot 17,\\
&{\bf\color{red}5\cdot 19},\quad \ldots
\end{split}
\]
By Proposition \ref{prop:1}, $D_3$ has pattern $\{4,2\}$, hence the positions of the elements needed to delete occur periodically in $M_3$. The period is $5\times 6=30$, where $6$ is the length of the previous patter $\{4,2\}$.

From another view, we make $5$ copies of the pattern $\{4,2\}$. Say,
\[
4, 2, 4, 2, 4, 2, 4, 2, 4, 2.
\]
After deleting $5$ and $25$, we get the string:
\[
(4+2), 4, 2, 4, 2, 4, (2+4), 2.
\]
It is the pattern of $D_5$, we denote it by $\mathcal{P}_{5}=\{6,4,2,4,2,4,6,2\}$. Hence the period is $T_5=8=2\cdot 5-2=(5-1)\cdot T_3$.

Hence, we have seen that we give the definition of $\mathcal{P}_{p_n}$ and $T_{p_n}$ inductively. When we prove that $D_{p_n}$ contains a minimum subset $\mathcal{P}_{p_n}$ which occurs periodically, we will prove that $D_{p_{n+1}}$ also contains a minimum subset which is denoted by $\mathcal{P}_{p_{n+1}}$ that occurs periodically. Then, we can define the period $T_{p_{n+1}}$.

To see this fact more clearly, we now show that the pattern of $D_7$ is
\begin{equation}\label{pattern:D_7}
\begin{split}
&10=(6,4),2,4,2,4,6,2;\\
&6,4,2,4,6=(2,4),6,2;\\
&6,4,2,6=(4,2),4,6,8=(2,6);\\
&\ 4,2,4,2,4,8=(6,2);\\
&6,4,6=(2,4),2,4,6,2;\\
&6,6=(4,2),4,2,4,6,2;\\
&6,4,2,4,2,10=(4,6),2.
\end{split}
\end{equation}
The period is $T_7=48=(7-1)\cdot(5-1)\cdot(3-1)\cdot(2-1)$.

In fact, $M_7$ is obtained by deleting the multiples of $7$ in $M_5$. We can rearrange the numbers in Figure \ref{fig:M5} as in Figure \ref{fig:M5-arrange}.

\begin{figure}
\centering
\begin{tabular}{|c|c|c|c|c|c|c|c|c|c|c|}
\hline
1& 7 & 11 & 13 & 17 & 19 & 23 & 29 & 31 & {\it 37} & {\it 41} \\
\hline
 & {\it 43} & {\it 47} & {\it 49} & {\it 53} & {\it 59} & {\it 61}& {\bf 67} & {\bf 71} & {\bf 73} & {\bf 77} \\
 & {\bf 79} & {\bf 83} & {\bf 89} & {\bf 91} & {\it 97} & {\it 101}&{\it 103} &{\it 107}&{\it 109} &{\it 113} \\
 &{\it 119} &{\it 121} &{\bf 127} &{\bf 131} &{\bf 133} &{\bf 137} &{\bf 139} &{\bf 143}&{\bf 149} & {\bf 151} \\
 &{\it 157} &{\it 161} &{\it 163} &{\it 167} &{\it 169} &{\it 173} &{\it 179} &{\it 181}&{\bf 187} & {\bf 191} \\
 &{\bf 193} &{\bf 197} &{\bf 199} &{\bf 203} & {\bf 209} & {\bf 211}& & & & \\
\hline
\hline
 & 217 & 221 & 223 & 227 & 229 & 233 & 239 & 241 & 247 & 251 \\
 & 253 & 257 & 259 & 263 & 269 & 271 & 277 & 281 & 283 & 287 \\
 & 289 & 293 & 299 & 301 & 307 & 311 & 313 & 317 & 319 & 323 \\
 & 329 & 331 & 337 & 341 & 343 & 347 & 349 & 353 & 359 & 361 \\
 & 367 & 371 & 373 & 377 & 379 & 383 & 389 & 391 & 397 & 401 \\
 & 403 & 407 & 409 & 413 & 419 & 421 & & & & \\
\hline
\hline
 & 427 & 431 & 433 & 437 & 439 & 443 & 449 & 451 & 457 & 461 \\
 & 463 & 467 & 469 & 473 & 479 & 481 & 487 & 491 & 493 & 497 \\
 & 499 & 503 & 509 & 511 & 517 & 521 & 523 & 527 & 529 & 533 \\
 & 539 & 541 & 547 & 551 & 553 & 557 & 559 & 563 & 569 & 571 \\
 & 577 & 581 & 583 & 587 & 589 & 593 & 599 & 601 & 607 & 611 \\
 & 613 & 617 & 619 & 623 & 629 & 631 & & & & \\
\hline
\hline
 & $\cdots$ & & & & & & & & & \\
\end{tabular}
\caption{rearranging $M_5$}
\label{fig:M5-arrange}
\end{figure}

We need to delete the multiples of $7$ in $M_5$ (See Figure \ref{fig:M7}). They are
\[
\begin{split}
&{\bf\color{red}7\cdot 1},\quad 7\cdot 7,\quad 7\cdot 11,\quad 7\cdot 13,\quad 7\cdot 17,\quad 7\cdot 19,\quad 7\cdot 23,\quad 7\cdot 29,\\
&{\bf\color{red}7\cdot 31},\quad 7\cdot 37,\quad 7\cdot 41,\quad 7\cdot 43,\quad 7\cdot 47,\quad 7\cdot 49,\quad 7\cdot 53,\quad 7\cdot 59,\\
&{\bf\color{red}7\cdot 61},\cdots
\end{split}
\]

Since we have proved that $D_5$ has period pattern $\{6,4,2,4,2,4,6,2\}$, we can assert that the elements to be deleted above are periodic. Every first item in the period has gap $7\times 30$, where $30$ is exactly the length of previous pattern $\{6,4,2,4,2,4,6,2\}$. Therefore, the pattern is obtained in the following way. First copy the previous pattern $7$ times, then combine some pairs of consecutive numbers to get a new pattern.
\[
\begin{aligned}
\underline{6,4},2,4,2,4,6,2\\
6,4,2,4,\underline{2,4},6,2\\
6,4,2,\underline{4,2},4,6,\underline{2}\\
\underline{6},4,2,4,2,4,\underline{6,2}\\
6,4,\underline{2,4},2,4,6,2\\
6,\underline{4,2},4,2,4,6,2\\
6,4,2,4,2,\underline{4,6},2\\
\end{aligned}
\]
Thus, we get the pattern of $D_7$, see \eqref{pattern:D_7}.

\begin{figure}[htbp]
\centering
\begin{tabular}{|c|c|c|c|c|c|c|c|c|c|c|}
\hline
1& {\bf\color{red}\sout{7}} & 11 & 13 & 17 & 19 & 23 & 29 & 31 & {\it 37} & {\it 41} \\
\hline
 & {\it 43} & {\it 47} & {\bf\color{blue}\sout{49}} & {\it 53} & {\it 59} & {\it 61}& {\bf 67} & {\bf 71} & {\bf 73} & {\bf\color{blue}\sout{77}} \\
 & {\bf 79} & {\bf 83} & {\bf 89} & {\bf\color{blue}\sout{91}} & {\it 97} & {\it 101}&{\it 103} &{\it 107}&{\it 109} &{\it 113} \\
$M_7^{(0)}$ &{\bf\color{blue}\sout{119}} &{\it 121} &{\bf 127} &{\bf 131} &{\bf\color{blue}\sout{133}} &{\bf 137} &{\bf 139} &{\bf 143}&{\bf 149} & {\bf 151} \\
 &{\it 157} &{\bf\color{blue}\sout{161}} &{\it 163} &{\it 167} &{\it 169} &{\it 173} &{\it 179} &{\it 181}&{\bf 187} & {\bf 191} \\
 &{\bf 193} &{\bf 197} &{\bf 199} &{\bf\color{blue}\sout{203}} & {\bf 209} & {\bf 211}& & & & \\
\hline
\hline
 & {\bf\color{red}\sout{217}} & 221 & 223 & 227 & 229 & 233 & 239 & 241 & 247 & 251 \\
 & 253 & 257 & {\bf\color{blue}\sout{259}} & 263 & 269 & 271 & 277 & 281 & 283 & {\bf\color{blue}\sout{287}} \\
 & 289 & 293 & 299 & {\bf\color{blue}\sout{301}} & 307 & 311 & 313 & 317 & 319 & 323 \\
$M_7^{(1)}$ & {\bf\color{blue}\sout{329}} & 331 & 337 & 341 & {\bf\color{blue}\sout{343}} & 347 & 349 & 353 & 359 & 361 \\
 & 367 & {\bf\color{blue}\sout{371}} & 373 & 377 & 379 & 383 & 389 & 391 & 397 & 401 \\
 & 403 & 407 & 409 & {\bf\color{blue}\sout{413}} & 419 & 421 & & & & \\
\hline
\hline
 & {\bf\color{red}\sout{427}} & 431 & 433 & 437 & 439 & 443 & 449 & 451 & 457 & 461 \\
 & 463 & 467 & {\bf\color{blue}\sout{469}} & 473 & 479 & 481 & 487 & 491 & 493 & {\bf\color{blue}\sout{497}} \\
 & 499 & 503 & 509 & {\bf\color{blue}\sout{511}} & 517 & 521 & 523 & 527 & 529 & 533 \\
$M_7^{(2)}$ & {\bf\color{blue}\sout{539}} & 541 & 547 & 551 & {\bf\color{blue}\sout{553}} & 557 & 559 & 563 & 569 & 571 \\
 & 577 & {\bf\color{blue}\sout{581}} & 583 & 587 & 589 & 593 & 599 & 601 & 607 & 611 \\
 & 613 & 617 & 619 & {\bf\color{blue}\sout{623}} & 629 & 631 & & & & \\
\hline
\hline
 & $\cdots$ & & & & & & & & & \\
\end{tabular}
\caption{Obtain $M_{7}$ from $M_5$ by deleting the $7h$, $h\in M_5$}
\label{fig:M7}
\end{figure}

If we construct $M_{11}$, then we will consider the following eleven blocks:
\[
M_7^{(0)},M_7^{(1)},M_7^{(2)},\ldots,M_7^{(10)},
\]
where we use $M_{7}^{(0)}$ to denote the first block in $M_7$ corresponding to the first period of $D_{7}$.


Generally, when constructing the new pattern for $D_{p_{n+1}}$ from $D_{p_n}$, we first take $p_{n+1}$ copies of previous pattern for $D_{p_n}$. Then the period obeys the recursive formula:
\[
T_{p_{n+1}}=p_{n+1}\cdot T_{p_n}-T_{p_n}=(p_{n+1}-1)\cdot T_{p_n}.
\]
According to the induction hypothesis,
\[
T_{p_n}=\prod_{i=1}^{n}(p_i-1),
\]
we have
\[
T_{p_{n+1}}=\prod_{i=1}^{n+1}(p_i-1).
\]
\end{proof}

\begin{cor}
$L(\mathcal{P}_{p_n})=\prod_{i=1}^n p_i$.
\end{cor}
\begin{proof}
Using $L(\mathcal{P}_{p_n})=p_nL(\mathcal{P}_{p_{n-1}})$.
\end{proof}

\begin{lem}
If $p_n\geqslant 11$ (i.e., $n\geqslant 5$), $p_n^2$ is contained in the first block of $M_{p_{n-1}}$. (Here the block means the subset of $M_{p_{n-1}}$ corresponding to the $D_{p_{n-1}}$.)
\end{lem}
\begin{proof}
First,
\[
\begin{split}
&p_n^2< p_n+L(\mathcal{P}_{p_{n-1}}) \\
\Leftrightarrow\  & p_n(p_n-1)< p_{n-1}\cdot p_{n-2}\cdots 5\cdot 3\cdot 2.\\
\end{split}
\]
It is easy to check that $p_k=11$ satisfies the inequality:
\[
p_k(p_k-1)< p_{k-1}\cdot p_{k-2}\cdots 5\cdot 3\cdot 2.
\]
By Bertrand's postulate $p_{k+1}<2p_k$, then it is easy to show that
\[
p_{k+1}(p_{k+1}-1)< p_{k}\cdot p_{k-1}\cdot p_{k-2}\cdots 5\cdot 3\cdot 2
\]
also holds if $k\geqslant 5$.

Then we have completed the proof.
\end{proof}

Hence, if $n\geqslant 5$, then the numbers between $p_n$ and $p_n^2$ which belong to the first block of $M_{p_{n-1}}$ are all primes. In fact, we observe that for any $n\geqslant 1$, for any element $q(>1)$ in the the first block of $M_{p_{n-1}}$, if $q$ satisfies $p_n<q<p_n^2$, then it is a prime.


\begin{thm}\label{thm:mirror-symmetry}
The gap sequence except the last element in the pattern is mirror symmetric.
\end{thm}
\begin{proof}
For example, when constructing $D_{11}$ from $D_7$, we delete the numbers $\{11k\mid k\in M_{7}\}$. We list the divisors $k$ of the first period. They are
\[
\begin{matrix}
 1 &     11&    13&    17&        19&    23&    29&    31&    37&    41\\
 43&     47&    53&    59&        61&    67&    71&    73&    79&    83\\
 89&     97&  101&  {\bf 103} &  *  {\bf 107} &  109&  113&  121&  127&  131\\
 137&   139&  143&  149&      151&  157&  163&  167&  169&  173\\
 179&   181&  187&  191&      193&  197&  199&  209&     &     \\
\end{matrix}
\]
These $48$ numbers are mirror complemental. The sum of every two symmetric numbers about $*$ are equal to $210$. More precisely, every block in $M_7$ are mirror complemental about the center of the block. It infers that the mirror symmetry of the gap sequence except the last element in the pattern $\mathcal{P}_{11}$. Therefore, when constructing $D_{p_{n+2}}$ from $D_{p_{n+1}}$, we conclude that the numbers which to be deleted $\{p_{n+2}k\mid k\in M_{p_{n+1}}\}$ are mirror complemental because the numbers $k\in M_{p_{n+1}}$ corresponding the first period are mirror complemental by the inductive hypothesis. And by inductive hypothesis, the first $p_{n+2}$ blocks in $M_{p_{n+1}}$ are mirror complemental. Therefore, the gap sequence in the pattern $\mathcal{P}_{p_{n+2}}$ except the last number is mirror symmetric.
\end{proof}

By the argument above, we easily have
\begin{cor}\label{cor:1}
The last element in every pattern $\mathcal{P}_{p_n}$ is $2$.
\end{cor}
\begin{proof}
We still illustrate this fact first by the above example. Where the length of the period is $210$. Thus the next number of $209$ in the second block of $M_7$ is $211$. Thus the last element in $\mathcal{P}_7$ is $2=211-209$.

Generally, if we use $x$ to denote the last element in the first block of $M_{p_n}$, then $1+x=L(\mathcal{P}_{p_n})$ by the mirror complementary. On the other hand, the first element denoted by $y$ in the second block of $M_{p_n}$ satisfies that $y-1=L(\mathcal{P}_{p_n})$. Thus, $y-x=2$ which is just the last element of the patter $\mathcal{P}_{p_n}$.
\end{proof}

\begin{thm}\label{thm:about-skips}
(i) The multiplicities of skips $2$ and $4$ are always same and odd.\\
(ii) The central gap is always $4$, and the multiplicities of all gaps except $2$ and $4$ are even.
\end{thm}
\begin{proof}
(i) In fact, these multiplicities for $k$-prime basis, $k>1$, are both precisely equal to
\begin{equation}\label{eqn:tpk}
(3-2)\cdot(5-2)\cdots(p_k-2),
\end{equation}
which is always odd as well. For example, there are 15 skips $2$ in the pattern $\mathcal{P}_7$ (see \eqref{pattern:D_7}).

For element $x$ in $M_{p_k}$, we consider the residues of $2,3,5,7,\ldots,p_k$ respectively, i.e.,
\[
x\equiv(x_1,x_2,x_3,\ldots,x_k)\mod (p_1,p_2,p_3,\ldots,p_k),
\]
and for simplicity we write it as a vector $x=(x_1,x_2,x_3,\ldots,x_k)$.

The multiplicity of skip $2$ is easily obtained by noting that some integer $x = (x_1, x_2,\ldots, x_k)$ is the first relatively prime (r-prime) before skip $2$, iff $x$ has residues of the form $(1, x_2, x_3,\ldots, x_k)$ where no $x_i$ is $0$ or $-2$. (If $x_i=0$, then $x$ is a multiple of $p_i$ which will not in $M_{p_k}$ for $i\leqslant k$. Similarly, if $x_i=-2$, then $x_i+2$ is a multiple of $p_i$.) Since number of such $k$-tuples with $2$ forbidden values on the upper $k-1$ residues is $(3-2)\cdot(5-2)\cdots(p_k-2)$, that is the multiplicity of skip $2$.

Similarly, for skip $4$, the integer $x$ is the first r-prime before skip $4$ iff it has residues of the form $(1, x_2, x_3,\ldots,x_k)$ where no $x_i$ is $0$ or $-4$, yielding the above multiplicity with the two forbidden values on $x_2, x_3,\ldots, x_k$.

Note that the above simple method doesn't work for skips larger than $4$ since these can be obtained both via replication (by increasing k) and via merging of adjacent skips which follows replication. In contrast, for $k>1$, skips $2$ and $4$ can be obtained only via replication (with growing $k$), but not via skip merging during filtering, since skip subsequences $2, 2$ cannot occur for $k>1$ (the replicated residue pattern for $k>1$ is $2,4,2,4,\ldots$).

(ii) Regarding the formula above \eqref{eqn:tpk} for multiplicities of gaps $2$ and $4$, the second statement of (ii) (the multiplicities of all gaps except $2$ and $4$ are even) is true. This follows easily from the mirror symmetry (which is easiest seen as symmetry gaps with respect to change of signs of residues) and the fact that the central gap (which could straddle the symmetry center and mirror into itself) is always 4. (End gap is always 2, and these are the only 2 exceptions.) Regarding the latter clause, for a k-prime basis $p_1=2, p_2=3, \ldots, p_k$, the central integer (symmetry axis) is at $C=3\cdot 5\cdots p_k$ and its residue pattern is $C=(1, 0, 0,\ldots, 0)$. The integers around $C$ are then in residue form as follows:
\[
\begin{array}{rrrrrcrl}
C-2: & -1 & -2 & -2 & -2 & \ldots & -2  & \text{r-prime}\\
C-1: &  0 & -1 & -1 & -1 & \ldots & -1  & \text{r-composite}\\
  C: &  1 & 0  & 0  & 0  & \ldots & 0   & \text{r-composite}\\
C+1: &  0 & 1  & 1  & 1  & \ldots & 1   & \text{r-composite}\\
C+2: &  1 & 2  & 2  & 2  & \ldots & 2   & \text{r-prime}\\
\end{array}
\]

which always has $2$ r-primes at $C\pm 2$, and $3$ r-composites at $C, C\pm 1$, yielding thus gap $4$ at the symmetry axis. The mirror symmetry around $C$ then implies the second statement.
\end{proof}

\begin{remark}\label{rem:largest-skip}
The largest skip comes in pairs only up to first $8$ basis primes (i.e. up to $\mathcal{P}_{19}$), but already for the $9$ prime basis ($\mathcal{P}_{23}$) the largest skip is $40$ which occurs $12$ times. Beyond the $9$ prime basis, the longest pattern occurs in pairs only for $3$ more cases: $10$, $11$ and $13$ basis primes i.e. $\mathcal{P}_{29}$, $\mathcal{P}_{31}$ and $\mathcal{P}_{41}$. All other cases have larger multiplicities of the longest skip.

Note also that for the basis with first $k$ primes $p_1=2, p_2=3, \ldots, p_k$, the trivial low bound for the max skip is $2p_{k-1}$, e.g. for the set $\mathcal{P}_{13} (k=6)$, the low bound is $2\cdot 11=22$ which happens to be the actual longest skip for $\mathcal{P}_{13}$.

The general residue pattern for solution meeting this low bound with $k$ prime basis is $( 0, 0,  0, ... 0, \mp 1, \pm 1 )$, where first $k-2$ residues are $0$, and last two are $+1$ and $-1$ or $-1$ and $+1$ (which corresponds to the two solutions). The two longest patterns occur only when this low bound $2p_{k-1}$ is met, otherwise there are more than $2$ maxima (again there is even number of solutions due to mirror symmetry).

For the $9$ prime basis ($\mathcal{P}_{23}$), the above low bound is $2\cdot 19=38$, while the max skip is $40$, hence the low bound defect is $40-38=2$. The max skip is $100$ for $\mathcal{P}_{47}$, while the low bound is $2\cdot 43=86$, yielding defect $14$. The max skip is $200$ for $\mathcal{P}_{79}$ while the low bound is $2\cdot 73=146$, yielding defect $54$. The max skip is $300$ for $\mathcal{P}_{107}$ ($28$ primes basis), while low bound is $2\cdot 103=206$, defect $94$. The max skip is $414$ (i.e. the first $\geqslant 400$) for $\mathcal{P}_{139}$ ($34$ primes basis), while low bound is $2\cdot 137=274$, defect $140$. The max skip is $510$ (the first $\geqslant 500$) for $\mathcal{P}_{167}$  ($39$ primes basis), defect $184$, etc.
\end{remark}

\begin{conj}
$6$ is the number that occurs the most times in any pattern $\mathcal{P}_{p_n}$ when $n\geqslant 11$.
\end{conj}

\begin{remark}
Let $t_{p_{n+1}}$ denote the multiplicity of $2$ in the pattern for $D_{p_{n+1}}$, then we get a recurrence inequality,
\[
t_{p_{n+1}}\geqslant t_{p_{n}}\cdot p_{n+1}-T_{p_n}.
\]
But it is a rough estimate. In fact, by equation \eqref{eqn:tpk}, we have
\[
t_{p_{n+1}}>T_{p_n}.
\]
\end{remark}

\section{Application to get consecutive primes}

Based on the above discussion, to obtain the pattern for $D_{p_{n}}$, we first take $p_{n}$ copies of the previous pattern for $D_{p_{n-1}}$. Then, delete the corresponding elements $p_n h$ in $M_{p_{n-1}}$, where $h\in M_{p_{n-1}}$. These elements to be deleted are distributed in the following way. The number of them is $T_{p_{n-1}}$.

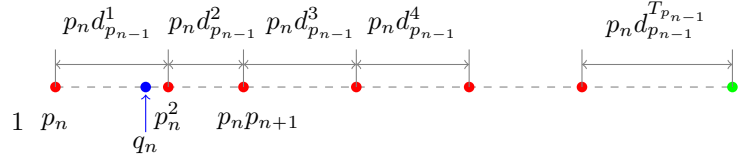
\begin{center}
\begin{tikzpicture}[scale=1]
\draw[gray,dashed] (0,0) -- (9,0);
\fill [red] ($(0,0)$) circle (2pt);
\node  at (-0.5,-0.7)[anchor=south] {$1$};
\node  at (0,-0.7)[anchor=south] {$p_n$};

\node  at (0.7,0.5)[anchor=south] {$p_nd_{p_{n-1}}^1$};

\fill [red] ($(1.5,0)$) circle (2pt);
\node  at (1.5,-0.7)[anchor=south] {$p_n^2$};

\fill [red] ($(2.5,0)$) circle (2pt);
\node  at (2.7,-0.7)[anchor=south] {$p_n p_{n+1}$};
\node  at (2.1,0.5)[anchor=south] {$p_nd_{p_{n-1}}^2$};
\node  at (3.4,0.5)[anchor=south] {$p_nd_{p_{n-1}}^3$};
\node  at (4.75,0.5)[anchor=south] {$p_nd_{p_{n-1}}^4$};

\node  at (8,0.5)[anchor=south] {$p_nd_{p_{n-1}}^{T_{p_{n-1}}}$};

\fill [red] ($(4.0,0)$) circle (2pt);
\fill [red] ($(5.5,0)$) circle (2pt);
\fill [red] ($(7.0,0)$) circle (2pt);
\fill [green] ($(9.0,0)$) circle (2pt);

\draw[gray] (0,0) -- (0,0.5);
\draw[gray] (1.5,0) -- (1.5,0.5);
\draw[gray] (2.5,0) -- (2.5,0.5);
\draw[gray] (4,0) -- (4,0.5);
\draw[gray] (5.5,0) -- (5.5,0.5);
\draw[gray] (7,0) -- (7,0.5);
\draw[gray] (9,0) -- (9,0.5);

\draw[gray][<->] (0,0.3) -- (1.5,0.3);
\draw[gray][<->] (1.5,0.3) -- (2.5,0.3);
\draw[gray][<->] (2.5,0.3) -- (4.0,0.3);
\draw[gray][<->] (4.0,0.3) -- (5.5,0.3);
\draw[gray][<->] (7,0.3) -- (9,0.3);

\fill [blue] ($(1.2,0)$) circle (2pt);
\draw[blue][->] (1.2,-0.6) -- (1.2,-0.1);
\node  at (1.2,-1.0)[anchor=south] {$q_n$};
\end{tikzpicture}
\end{center}

Here $\mathcal{P}_{p_{n-1}}=\{d_{p_{n-1}}^1, d_{p_{n-1}}^2,\ldots, d_{p_{n-1}}^{T_{p_{n-1}}}\}$ is the pattern for $D_{p_{n-1}}$. The red points are just the elements to be deleted. The blue point $q_n$ denotes the biggest prime which is less than $p_n^2$. The last green point is also deleted. But it is regarded as in the next period.

\begin{prop}\label{prop:2}
The numbers in the set
\[
\{q\in M_{p_{n-1}}\mid p_n<q<p_n p_{n+1},\quad q\neq p_n^2\}
\]
are consecutive primes.
\end{prop}

If there exist at least one gap $2$ in the pattern $\mathcal{P}_{p_{n-1}}$ located in the subset corresponding to the interval $[p_{n+1},q_n]$. Then, this gap $2$ is kept in the next pattern $\mathcal{P}_{p_n}$, and also in the subset corresponding to the interval $[p_{n+1},q_n]$.

%
%
%
%
%

\begin{thm}\label{thm:chi-interval}
Let $c$ be a positive integer. Define a characteristic function for each interval $I_k=[(k-1)c,kc)$, $k=1,2,\ldots$ as follows:
\[
\chi_{_{I_k}}=\begin{cases}
1, & \text{if}\ I_k\ \text{contains one or more primes},\\
0, & \text{otherwise}.\\
\end{cases}
\]
If $m$ is large enough, and if $c\geqslant 3$, then
\[
\frac{1}{c}\pi(mc)\leqslant\sum_{k=1}^{m}\chi_{_{I_k}}<\pi(mc),
\]
that is,
\[
\sum_{k=1}^{m}\chi_{_{I_k}}\asymp\frac{m}{\ln m}.
\]
\end{thm}
\begin{proof}
Let $J_{ki}, i=1,2,\ldots,c$ be the subintervals in $I_k$, each $J_{ki}$ has the form $[d,d+1)$. Then,
\[
\chi_{_{I_k}}=\max_{1\leqslant i\leqslant c}\chi_{_{J_{ki}}}.
\]
Thus,
\[
\sum_{k=1}^{m}\chi_{_{I_k}}=\sum_{k=1}^{m}\max_{i}\chi_{_{J_{ki}}}\geqslant\sum_{k=1}^{m}\frac{1}{c}\sum_{i=1}^{c}\chi_{_{J_{ki}}}=\frac{1}{c}\pi(mc).
\]
On the other hand, it is obvious that
\[
\sum_{k=1}^{m}\chi_{_{I_k}}<\pi(mc)\quad\text{for}\ c\geqslant 3.
\]

%
\end{proof}

Yitang Zhang \cite{ZhangYitang} proved the following result for consecutive primes based on the recent work of Goldston, Pintz and Y{\i}ld{\i}r{\i}m \cite{GPY2009A,GPY2010A} on the small gaps between consecutive primes.
\begin{equation}\label{eqn:zhang_result}
\liminf_{n\rightarrow\infty}(p_{n+1}-p_n)< H,
\end{equation}
where $H=7\times 10^7$. Then it was reduced to $4680$ by Polymath project \cite{ Granville,Polymath}. In late 2013, James Maynard and Terry Tao found a much simpler proof of Zhang's result giving $H=600$. A further progress based on this work has reduced $H$ to $252$ by Polymath project\cite{Green}. As of April 14, 2014, one year after Zhang's announcement, according to the Polymath project wiki, the bound has been reduced to $246$.

Hence, for $c\geqslant 246$, we have

\begin{prop}
There are infinitely many pairs $m_2>m_1>0$, $m_2$ and $m_1$ are large enough, such that
\[
\frac{1}{c}(\pi(m_2 c)-\pi(m_1 c))\leqslant\sum_{k=m_1}^{m_2}\chi_{_{I_k}}< \pi(m_2 c)-\pi(m_1 c).
\]
\end{prop}

\begin{conj}
For any $m_2>m_1>0$, $m_2$ and $m_1$ are large enough, and any $c>2$, we have
\[
\frac{1}{c}(\pi(m_2 c)-\pi(m_1 c))\leqslant\sum_{k=m_1}^{m_2}\chi_{_{I_k}}< \pi(m_2 c)-\pi(m_1 c).
\]
\end{conj}

%
%
%


\noindent{\bf Acknowledgments :}
We express our gratitude to Mr. Ratko V. Tomic (a researcher in Infinetics Technologies, Inc) who pointed out the mirror symmetry of the gap sequence which is stated in Theorem \ref{thm:mirror-symmetry}. And also he gives the proof of Theorem \ref{thm:about-skips} and makes the remark \ref{rem:largest-skip}. We thank him for discussing the question and introducing us the ideas involved.


\bigskip

\noindent Haifeng Xu\\
School of Mathematical Science\\
Yangzhou University\\
Jiangsu China 225002\\
hfxu@yzu.edu.cn\\
\medskip

\noindent Zuyi Zhang\\
Department of Mathematics\\
Suqian College\\
Jiangsu China 225002\\
zhangzuyi1993@hotmail.com\\
\medskip

\noindent Jiuru Zhou\\
School of Mathematical Science\\
Yangzhou University\\
Jiangsu China 225002\\
zhoujr1982@hotmail.com\\
\medskip


\end{document}